\documentclass{baustms}
\usepackage{url}

\DeclareMathOperator{\Isom}{Isom}
\DeclareMathOperator{\Ax}{Ax}
\DeclareMathOperator{\ax}{ax}

\theoremstyle{bams}
\newtheorem{thm}{Theorem}[section]
\newtheorem{prop}[thm]{Proposition}
\newtheorem{cor}[thm]{Corollary}
\newtheorem{lem}[thm]{Lemma}

\theoremstyle{bamsdefn}

\newtheorem{rem}[thm]{Remark}

\begin{document}
\runningtitle{Commensurability of Kleinian groups}
\title{Limit sets and commensurability of Kleinian groups}
\cauthor 
\author[1]{Wen-yuan Yang}
\address[1]{College of Mathematics and Econometrics, Hunan
University, Changsha, Hunan 410082 People's Republic of China
\\U.F.R. de Mathematiques, Universite de Lille 1,59655 Villneuve
D'Ascq Cedex, France \email{wyang@math.univ-lille1.fr}}
\author[2]{Yue-ping Jiang}
\address[2]{College of Mathematics and Econometrics, Hunan
University, Changsha, Hunan 410082 People's Republic of China
\email{ypjiang731@163.com}}


\authorheadline{W-Y. YANG and Y-P. JIANG}


\support{The first author is supported by the China-funded
Postgraduates Studying Aboard Program for Building Top University.
The second author is supported by National Natural Science
Foundation of China (No. 10671059) and Doctorate Foundation of the
Ministry of Education of China (No. 20060532023).}

\begin{abstract}
In this paper, we obtain several results on the commensurability of
two Kleinian groups and their limit sets. We prove that two finitely
generated subgroups $G_1$ and $G_2$ of an infinite co-volume
Kleinian group $G \subset \Isom(\mathbf{H}^3)$ having $\Lambda(G_1) =
\Lambda(G_2)$ are commensurable. In particular, it is proved that
any finitely generated subgroup $H$ of a Kleinian group $G \subset
\Isom(\mathbf{H}^3)$ with $\Lambda(H) = \Lambda(G)$ is of finite
index if and only if $H$ is not a virtually fiber subgroup.
\end{abstract}


\maketitle

\section{Introduction}\label{sec:1}
Two groups $G_1$ and $G_2$ are commensurable if their intersection
$G_1 \cap G_2$ is of finite index in both $G_1$ and $G_2$. In this
paper, we investigate the following question asked by J. Anderson
\cite{bm}: namely, if $G_1, G_2 \subset \Isom(\mathbf{H}^n)$ are
finitely generated and discrete, does $\Ax(G_1) = \Ax(G_2)$ imply that
$G_1$ and $G_2$ are commensurable? Here we use $\Ax(G)$ to denote the set
of axes of the hyperbolic elements of $G \subset
\Isom(\mathbf{H}^n)$.

The question has been discussed by several
authors. In 1990, G. Mess \cite{gm} showed that if $G_1$ and $G_2$
are non-elementary finitely generated Fuchsian groups having the
same nonempty set of simple axes, then $G_1$ and $G_2$ are
commensurable. Using some technical results on arithmetic
Kleinian groups, D. Long and A. Reid \cite{da} gave an affirmative
answer to this question in the case where $G_1$ and $G_2$ are arithmetic
Kleinian groups. Note that all the confirmed cases for the question
are geometrically finite groups. So it is natural to ask if the
question is true with the assumption that $G_1$ and $G_2$ are
geometrically finite. Recently, P. Susskind \cite{ps} constructed
two geometrically finite Kleinian groups in $\Isom(\mathbf{H}^n)$ (for $n
\geq 4$) having the same action on some invariant 2-hyperbolic plane
but whose intersection is infinitely generated. So this implies that
these two geometrically finite groups are not commensurable although
they have the same axes set. But it is worth pointing that these two
geometrically finite Kleinian groups generate a non-discrete group.
This example suggests that some additional conditions need to be
imposed to eliminate such `bad' groups.

In higher dimensions, we have the following consequence of P.
Susskind and G. Swarup's results \cite{ss} on the limit set of the
intersection of two geometrically finite Kleinian groups.

\begin{prop} \label{prop:geocomm}
Let $G_1$ and $G_2$ be two non-elementary geometrically finite subgroups
of a Kleinian group $G \subset \Isom(\mathbf{H}^n)$. Then $G_1$ and
$G_2$ are commensurable if and only if the limit sets $\Lambda(G_1)$ and
$\Lambda(G_2)$ are equal. In particular, $\Lambda(G_1) = \Lambda(G_2)$ if and only
if $\Ax(G_1) = \Ax(G_2) = \Ax(G_1 \cap G_2)$.
\end{prop}

In several cases, the condition that both subgroups lie in a larger
discrete group can be dropped.

\begin{cor} \label{cor:geocomm}
Let $G_1, G_2 \subset \Isom(\mathbf{H}^n)$ be two non-elementary
geometrically finite Kleinian groups of the second kind leaving no
$m$ hyperbolic planes invariant for $m < n-1$. Then $G_1$ and $G_2$
are commensurable if and only if $\Lambda(G_1) = \Lambda(G_2)$.
\end{cor}

In dimension 3, we can refine the analysis of the limit sets
using Anderson's results \cite{aj2} to get the following result in the
essence of the recent solution of the Tameness Conjecture (see
\cite{ag} and \cite{cag}) which states that all finitely generated
Kleinian groups in $\Isom(\mathbf{H}^3)$ are topologically tame.

\begin{thm} \label{thm:tamecomm}
Let $G_1$, $G_2$ be two non-elementary finitely generated subgroups of
an infinite co-volume Kleinian group $G \subset \Isom(\mathbf{H}^3)$.
Then $G_1$ and $G_2$ are commensurable if and only if $\Lambda(G_1)
= \Lambda(G_2)$. In particular, $\Lambda(G_1) = \Lambda(G_2)$ if and
only if $\Ax(G_1) = \Ax(G_2) = \Ax(G_1 \cap G_2)$
\end{thm}

Similarly, in the following case we are able to remove the ambient
discrete group.

\begin{cor} \label{cor:tamecomm}
Let $G_1, G_2 \subset \Isom(\mathbf{H}^3)$ be two non-elementary
finitely generated Kleinian groups of the second kind whose limit sets
are not circles. Then $G_1$ and $G_2$ are commensurable if and only
if $\Lambda(G_1) = \Lambda(G_2)$.
\end{cor}

In fact, under the hypotheses of the above results, the condition of having
the same limit sets of two Kleinian subgroups exactly implies having
the same axes sets. But Anderson's original formulation of the question
is only to suppose that two Kleinian groups have the same axes set.
So it is interesting to explore whether there is some essential
difference between the limit set and axes set. The following theorem
is a result in this direction, suggesting that the
`same axes set' condition is necessary in general for Anderson's
question.

\begin{thm} \label{thm:axes}
Let $H$ be a non-elementary finitely generated subgroup of a Kleinian
group $G \subset \Isom(\mathbf{H}^3)$. Suppose that $\Lambda(H) =
\Lambda(G)$. Then $[G:H]$ is finite if and only if $G$ is not virtually
fibered over $H$. In particular, $[G:H]$ is finite if and only if $\Ax(H)
= \Ax(G)$.
\end{thm}
\begin{rem}
We remark that the case of $H$ being geometrically finite is proved
by P. Susskind and G. Swarup \cite[Theorem 1]{ss}. Theorem
\ref{thm:axes} actually proves the special case of Anderson's question where
$G_1$ is a subgroup of $G_2$.
\end{rem}



The paper is organized as follows. In section \ref{sec:2}, we collect some
results on the limit set of the intersection of two Kleinian groups and
prove some useful lemmas for later use. In section \ref{sec:3}, we prove
Theorems \ref{thm:tamecomm} and \ref{thm:axes}.


\section{Preliminaries}\label{sec:2}

Let $\textbf{B}^n$ denote the closed ball $\mathbf{H}^n\cup
\mathbf{S}^{n-1}$, whose boundary $\mathbf{S}^{n-1}$ is identified
via stereographic projection with $\overline{\mathbf{R}}^{n-1} =
\mathbf{R}^{n-1} \cup \infty$. Let $\Isom(\mathbf{H}^n)$ be the full
group of isometries of $\mathbf{H}^n$ and let $G \subset
\Isom(\mathbf{H}^n)$ be a \textit{Kleinian} group; that is, a
discrete subgroup of $\Isom(\mathbf{H}^n)$. Then $G$ acts discontinuously
on $\mathbf{H}^n$ if and only if $G$ is discrete. Furthermore, $G$ acts on
$\mathbf{S}^{n-1}$ as a group of conformal homeomorphisms. The
\textit{set of discontinuity} $\Omega(G)$ of $G$ is the subset of
$\mathbf{S}^{n-1}$ on which $G$ acts discontinuously; the
\textit{limit set} $\Lambda(G)$ is the complement of $\Omega(G)$
in $\mathbf{S}^{n-1}$. A Kleinian group is said to be of the \textit{second
kind} if $\Omega(G)$ is nonempty otherwise it is said to be of the \textit{first
kind}.

The elements of  $\Isom(\mathbf{H}^n)$  are classified in terms of
their fixed point sets. An element $g \neq \text{id}$ in
$\Isom(\mathbf{H}^n)$ is $\textit{elliptic}$ if it has a fixed point
in $\mathbf{H}^n$, \textit{parabolic} if it has exactly one fixed
point which lies in $\mathbf{S}^{n-1}$, \textit{hyperbolic} if it
has exactly two fixed points which lie in $\mathbf{S}^{n-1}$. The
unique geodesic joining the two fixed points of the hyperbolic
element $g$, which is invariant under $g$, is called the
\textit{axis} of the hyperbolic element and is denoted by
$\ax(g)$. The limit set $\Lambda(G)$ is the closure of the set of fixed
points of hyperbolic and parabolic elements of $G$. A Kleinian group
whose limit set contains fewer than three points is called
$\textit{elementary}$ and is otherwise called
$\textit{non-elementary}$.

For a non-elementary Kleinian group $G$, define  $\widetilde{C}(G)$
to be the smallest nonempty convex set in $\mathbf{H}^n$ which is
invariant under the action of $G$; this is the \textit{convex hull}
of $\Lambda(G)$. The quotient $C(G) = \widetilde{C}(G)/G$ is the
\textit{convex core} of $M = \mathbf{H}^n/G$. The group $G$ is
\textit{geometrically finite} if the convex core $C(G)$ has finite
volume.

By Margulis's lemma, it is known that there is a positive constant
$\epsilon_0$ such that for any Kleinian group $G \in
\Isom(\mathbf{H}^n)$ and $\epsilon < \epsilon_0$, the part of
$\mathbf{H}^n/G$ where the injectivity radius is less than
$\epsilon$ is a disjoint union of tubular neighbourhoods of closed
geodesics, whose lengths are less than 2$\epsilon$, and cusp
neighbourhoods. In dimensions 2 and 3, these cusp neighbourhoods can
be taken to be disjoint quotients of horoballs by the corresponding
parabolic subgroup. This set of disjoint horoballs is called a
\textit{precisely invariant system of horoballs} for $G$. In
dimension 3, it is often helpful to identify the infinity boundary
$\mathbf{S}^2$ of $\mathbf{H}^3$ with the extended complex plane
$\overline{\mathbf C}$. In particular, the fixed point of a rank 1
parabolic subgroup $J$ of $G$ is called $\textit{doubly cusped}$ if
there are two disjoint circular discs $B_1$, $B_2 \subset
\overline{\mathbf C}$ such that $B_1\cup B_2$ is precisely invariant
under $J$ in $G$. In this case, the parabolic elements of $J$ are
also called doubly cusped.

In dimension 3, we call a Kleinian group $G$ \textit{topologically
tame} if the manifold $M = \mathbf{H}^3/G$ is homeomorphic to the
interior of a compact 3-manifold. Denote by $M^c$ the complement of
these cusp neighbourhoods. Using the relative core theorem of
McCullough \cite{mc}, there exists a compact submanifold $N$ of
$M^c$ such that the inclusion of $N$ in $M^c$ is a homotopy
equivalence, every torus component of $\partial(M^c)$ lies in $N$,
and $N$ meets each annular component of $\partial(M^c)$ in an
annulus. Call such an $N$ a \textit{relative compact core} for $M$.
The components of $\partial(N) -
\partial(M^c)$ are the \textit{relative boundary components} of $N$. The
\textit{ends} of $M^c$ are in one-to-one correspondence with the
components of $M^c-N$. An end $E$ of $M^c$ is \textit{geometrically
finite} if it has a neighbourhood disjoint from $C(G)$. Otherwise,
$E$ is \textit{geometrically infinite}.

A Kleinian group $G \subset \Isom(\mathbf{H}^3)$ is \textit{virtually
fibered} over a subgroup $H$ if there are finite index subgroups
$G^0$ of $G$ and $H^0$ of $H$ such that $\mathbf H^3/G^0$ has finite volume
and fibers over the circle with the fiber subgroup $H^0$. Note that
$H^0$ is then a normal subgroup of $G^0$, and so $\Lambda(H^0)$ =
$\Lambda(G^0)$ = $\mathbf{S}^2$.

In order to analyze the geometry of a geometrically infinite Kleinian
group, we will use Canary's covering theorem, which
generalizes a theorem of Thurston \cite{thurston}. Note that the
Tameness Theorem (\cite{ag} and \cite{cag}) states that all finitely
generated Kleinian groups in $\Isom(\mathbf{H}^3)$ are topologically
tame.
\begin{thm}[{\cite[The Covering Theorem]{ca}}] \label{thm:covering}
Let $G$ be a torsion free Kleinian group in $\Isom(\mathbf{H}^3)$ and
let $H$ be a non-elementary finitely generated subgroup of $G$. Let
$N = \mathbf{H}^3/G$, let $M = \mathbf{H}^3/H$, and let $p : M
\rightarrow N$ be the covering map. If $M$ has a geometrically
infinite end $E$, then either $G$ is virtually fibered over $H$ or
$E$ has a neighbourhood $U$ such that $p$ is finite-to-one on $U$.
\end{thm}

Now we list several results on the limit set of the
intersection of two Kleinian groups, which describe
$\Lambda(G_1\cap{G_2})$ in terms of $\Lambda(G_1)$ and
$\Lambda(G_2)$, where $G_1$ and $G_2$ are subgroups of a Kleinian
group $G$. Here we only collect the results used in this paper and
state them in an appropriate form for our purpose. See \cite{aj} for
a useful survey and the bibliography therein for the results in full
details.
\begin{thm}[{\cite[Theorem 3]{ss}}] \label{thm:ss}
Let $G_1$, $G_2$ be two geometrically finite subgroups of Kleinian
group $G \subset \Isom(\mathbf{H}^n)$. Then $\Lambda(G_1) \cap
\Lambda(G_2) = \Lambda(G_1\cap{G_2}) \cup P$ where $P$ consists of
some parabolic fixed points of $G_1$ and $G_2$.
\end{thm}

\begin{prop}[{\cite[Corollary 1]{ss}}] \label{prop:ss}
Let $H$ be geometrically finite and $j$  be a hyperbolic element
with a  fixed  point in $\Lambda(H)$. If  $\langle H,j\rangle$ is discrete, then
$j^n \in H$  for some $n > 0$.
\end{prop}

Based on the above results, we get the following lemma
characterizing the relationship between limit sets and axes sets.
\begin{lem} \label{lem:geo}
Let $G_1$, $G_2$ be two geometrically finite subgroups of Kleinian
group $G \subset \Isom(\mathbf{H}^n)$. Then $\Lambda(G_1) =
\Lambda(G_2)$ if and only if $\Ax(G_1) = \Ax(G_2) = \Ax(G_1 \cap G_2)$
\end{lem}

\begin{proof}
If $G_1$ and $G_2$ are geometrically finite subgroups of a Kleinian
group $G$, then $G_1\cap{G_2}$ is again geometrically
finite (\cite[Theorem 4]{ss}). It is well known that a hyperbolic
element cannot share one fixed point with a parabolic element in a
discrete group. Thus by applying Theorem \ref{thm:ss} to
$\Lambda(G_1)\cap \Lambda(G_2)$, we can conclude that any hyperbolic
element $h \in G_i$ has at least one fixed point in
$\Lambda(G_1\cap{G_2})$ for $i=1,2$. Now by Proposition
\ref{prop:ss}, we have $h^j \in G_1\cap{G_2}$ for some large integer
$j>0$. This implies the axis $\ax(h)$ of $h$ belongs to
$\Ax(G_1\cap{G_2})$. Therefore we have $\Ax(G_1) = \Ax(G_2) =
\Ax(G_1\cap{G_2})$ and it also follows that $P$ is the empty set.

The other direction is easy to see using the fact that the set of fixed
points of hyperbolic elements of $G$ is dense in $\Lambda(G)$.
\end{proof}

In dimension 3, J. Anderson \cite{aj2} carried out a more careful
analysis on the limit set of the intersection of two topologically tame
Kleinian groups. Combined with the recent solution of the Tameness
Conjecture (\cite{ag} and \cite{cag}), we have the following theorem.

\begin{thm}[{\cite[Theorem C]{aj2}}] \label{thm:an}
Let $G \subset\Isom(\mathbf{H}^3)$ be a Kleinian group, and let $G_1$ and
$G_2$ be non-elementary finitely generated subgroups of $G$, then
$\Lambda(G_1) \cap \Lambda(G_2) = \Lambda(G_1\cap{G_2}) \cup P$
where $P$ is empty or consists of some parabolic fixed points of
$G_1$ and~$G_2$.
\end{thm}

\begin{prop}[{\cite[Theorem A]{aj2}}] \label{prop:an}
Let $H$ be finitely generated Kleinian group and $j$  be a
hyperbolic element with a fixed point in $\Lambda(H)$. If  $\langle H,j\rangle$
is discrete, then either $\langle H,j\rangle$ is virtually fibered over $H$ or
$j^n \in H$ for some $n > 0$.
\end{prop}

Similarly, we obtain the following lemma.
\begin{lem} \label{lem:tame}
Let $G_1$ and $G_2$ be two non-elementary finitely generated subgroups of
an infinite co-volume Kleinian group $G \subset \Isom(\mathbf{H}^3)$.
Then $\Lambda(G_1) = \Lambda(G_2)$ if and only if $\Ax(G_1) = \Ax(G_2)
= \Ax(G_1 \cap G_2)$
\end{lem}

\begin{proof}
Observe that our hypothesis `$G$ is an infinite co-volume Kleinian
group' implies that $G$ is not virtually fibered over $G_1$. Since
the intersection of any pair of finitely generated subgroups of a
Kleinian group is finitely generated (see \cite{aj3}), we see that
$G_1\cap{G_2}$ is finitely generated. Using Theorem \ref{thm:an}
and Proposition \ref{prop:an}, we can argue exactly as in Lemma
\ref{lem:geo} to obtain $\Ax(G_1) = \Ax(G_2) = \Ax(G_1\cap{G_2})$, if
we suppose $\Lambda(G_1) = \Lambda({G_2})$.
\end{proof}
\begin{rem}
The condition of $G$ being infinite co-volume cannot be dropped, as
will be seen in the proof of Theorem \ref{thm:axes}. Namely, the
Kleinian group $G$ fibering over $H$ has a different axes set from that
of its fiber subgroup $H$.
\end{rem}

In the following two lemmas, we give some useful properties about
the same axes sets of two Kleinian groups.

\begin{lem} \label{lem:finite}
Let $G$ be a non-elementary finitely generated Kleinian group and
$H$ a subgroup of finite index in $G$. Then $\Ax(G) = \Ax(H)$.
\end{lem}
\begin{proof}
It is obvious that $\Ax(H) \subset \Ax(G)$. Conversely, since $[G:H]$
is finite, for any hyperbolic element $g$ with axis $\ax(g) \in
\Ax(G)$, there are two integers $i$ and $j$ such that $g^i H= g^j H$
and thus $g^{i-j} \in H$. It follows that $\ax(g) \in \Ax(H)$. The
proof is complete.
\end{proof}

\begin{rem}
In fact, our Theorem \ref{thm:axes} proves that the converse of the above
Lemma is also true when $H$ is a finitely generated subgroup of $G \subset
\Isom(\mathbf{H}^3)$.
\end{rem}

\begin{lem} \label{lem:burnside}
Let G be a non-elementary finitely generated, torsion free Kleinian
group and H be a subgroup of G. Suppose that $\Ax(G) = \Ax(H)$. Then
for every hyperbolic element $g \in G$, $g^n \in H$ for some $n>0$.
\end{lem}
\begin{proof}
For any hyperbolic element $g \in G$, we can choose a hyperbolic
element $h$ from $H$ such that $\ax(g)=\ax(h)$ by the hypothesis
$\Ax(G) = \Ax(H)$. It follows that the subgroup $\langle g,h\rangle$ is
elementary and torsion free. By the characterization of
elementary Kleinian groups it follows that $\langle g,h\rangle$ is actually a cyclic
subgroup $\langle f\rangle$ of $G$. Thus we can write $g=f^m$ and $h=f^n$ for two
appropriate integers $m$, $n$. Now we have found the integer $n$
such that $g^n = h^m \in H$, which proves the lemma.
\end{proof}

\section{Proofs}\label{sec:3}

\begin{proof}[Proof of Proposition \ref{prop:geocomm}]
Recall $G_1$ and $G_2$ are commensurable if the intersection $G_1
\cap G_2$ is of finite index in both $G_1$ and $G_2$. By Lemma
\ref{lem:finite}, we have $\Ax(G_1)=\Ax(G_2)$ and thus
$\Lambda(G_1)=\Lambda(G_2)$, if $G_1$ and $G_2$ are commensurable.
So it remains to prove the converse.

If $\Lambda(G_1)=\Lambda(G_2)$,  we have $\Ax(G_1) = \Ax(G_2) = \Ax(G_1
\cap G_2)$ by Lemma \ref{lem:geo}. Therefore it follows that
$\Lambda(G_1) = \Lambda(G_2) = \Lambda(G_1\cap{G_2})$, since the set
of fixed points of hyperbolic elements of $G$ is dense in
$\Lambda(G)$. Now we can conclude that $G_1\cap{G_2}$ is of finite
index in both $G_1$ and $G_2$, by using Theorem 1 in \cite{ss} which
states that any geometrically finite subgroup $H$ of a Kleinian
group $G$ is of finite index in $G$ if $\Lambda(H)=\Lambda(G)$.

The second assertion is just Lemma \ref{lem:geo}. This completes
the proof.
\end{proof}

\begin{proof}[Proof of Corollary \ref{cor:geocomm}]
It is well known that the stabilizer in $\Isom(\mathbf{H}^n)$ of the
limit set of a non-elementary Kleinian group $G$ of the second kind
leaving no $m$ hyperbolic planes invariant for $m < n-1$, is itself a
Kleinian group. See for example Greenberg \cite{lg}, where the discreteness
of the stabilizer of that limit set is proved.

Thus $\Ax(G_1) = \Ax(G_2)$ implies that $G_1$ and $G_2$ together lie
in a common Kleinian group, which is the stabilizer of the common
limit set of $G_1$ and $G_2$. Thus Proposition \ref{prop:geocomm} completes the proof.
\end{proof}

\begin{proof}[Proof of Theorem \ref{thm:tamecomm}]
Observe that the fundamental domain of the subgroup $H$ is the union of
translates of the fundamental domain of $G$ by left $H$-coset representatives in
$G$. So the subgroup $\langle G_1,G_2\rangle$ also has infinite co-volume, and we
can assume that $G$ is finitely generated by replacing $G$ by
$\langle G_1,G_2\rangle$. As the conclusion is easily seen to be unaffected by
passage to a finite index subgroup, we may use Selberg's Lemma to
pass to a finite index, torsion free subgroup of $G$. Hence, without loss of
generality, we may assume that $G$ is finitely generated and torsion free.

By Proposition \ref{prop:geocomm}, the conclusion is trivial if $G$ is
geometrically finite. So we suppose that $G$ is geometrically infinite. Then
$M = \textbf{H}^{3}/G$ has infinite volume. Let $C$ be a compact
core for $M$. Since $M$ has infinite volume, $\partial{C}$ contains
a surface of genus at least two. Then using Thurston's
geometrization theorem for Haken three-manifolds (see \cite{jm}),
there exists a geometrically finite Kleinian group with non-empty
discontinuity domain, which is isomorphic to $G$.

Now our task is to give an algebraic characterization of the limit
set of $G_1\cap{G_2}$ in $G_i$ such that the relationships between
$\Lambda(G_1\cap{G_2})$ and $\Lambda(G_i)$ can be passed to the ones
between target isomorphic groups under the above isomorphism of $G$.
Then the conclusion of Theorem \ref{thm:tamecomm} follows from
Proposition \ref{prop:geocomm}. We claim that for every element $g
\in G_1$, there exists an integer $k$ such that $g^k \in
G_1\cap{G_2}$.

Firstly, by Lemma \ref{lem:tame}, we obtain that $\Ax(G_1) = \Ax(G_2)
= \Ax(G_1 \cap G_2)$. So for any hyperbolic element $g \in G_1$, the
integer $k$ obtained in Lemma \ref{lem:burnside} is such that
$g^k \in G_1\cap{G_2}$. Now we consider the remaining parabolic
elements. Theorem B of \cite{aj2} says that if no nontrivial power of a
parabolic element $h \in G_1$ lies in $G_1\cap{G_2}$, then there
exists a doubly cusped parabolic element $f \in G_1\cap{G_2}$ with
the same fixed point $\xi$ as $h$. Normalizing their fixed point
$\xi$ to $\infty$, we can suppose that $f(z)=z+1$ and $h(z)=z+\tau$,
where $\text{Im}(\tau) \neq 0$. Since $f$ is doubly cusped in
$G_1\cap{G_2}$, then $\Lambda(G_1\cap{G_2}) \subset \{z:|\text{Im}(z)| <
c\}$, for some constant $c$. But on the other hand,
$\Lambda(G_1\cap{G_2})$ is also kept invariant under $h$, which
contradicts the fact that $\Lambda(G_1\cap{G_2})$ is invariant
under $f$. Therefore the claim is proved for all elements including
parabolic elements. A similar claim holds for $G_1\cap{G_2}$ and
$G_2$.

Under the isomorphism, using the above claims, we can conclude that the
limit set of the (isomorphic) image of $G_1 \cap G_2$ is equal to those
of the (isomorphic) images of $G_1$ and ${G_2}$. The proof is
complete as a consequence of Proposition \ref{prop:geocomm}.
\end{proof}

\begin{rem}
Theorem \ref{thm:tamecomm} can be thought as a geometric version of
Lemma 5.4 in \cite{aj2}, which uses an algebraic assumption on the limit
sets of the groups involved.
\end{rem}

\begin{proof}[Proof of Corollary \ref{cor:tamecomm}]
This is proved similarly to Corollary \ref{cor:geocomm}.
\end{proof}

\begin{proof}[Proof of Theorem \ref{thm:axes}]
In view of Lemma \ref{lem:finite}, we may assume, without loss of
generality, that $H$ is finitely generated and torsion free by using
Selberg's Lemma to pass to a finite index, torsion free subgroup of
$H$.

If $H$ is geometrically finite, then the conclusion follows from a
result of P. Susskind and G. Swarup \cite{ss}, which states that a
non-elementary geometrically finite subgroup sharing the same limit
set with the ambient discrete group is of finite index. So next we
suppose that $H$ is geometrically infinite. Then there exist
finitely many geometrically infinite ends $E_i$ for the manifold $N
:= \textbf{H}^{3}/H$.

We first claim that $\Ax(H) = \Ax(G)$ implies that $H$ cannot be a
virtually fiber subgroup of $G$. Otherwise, by taking finite index
subgroups of $G$ and $H$, we can suppose $H$ is normal in $G$. Then
it follows that every element of the quotient group $G/H$ has finite
order by Lemma \ref{lem:burnside}. Thus $G/H$ could not be
isomorphic to $\textbf{Z}$. This is a contradiction. So $H$ is not a
virtually fiber subgroup of $G$.

Using the Covering Theorem \ref{thm:covering}, we know that for
each geometrically infinite end $E_i$, there exists a neighbourhood
$U_i$ of $E_i$ such that the covering map $\mathcal {P}$: $N
\rightarrow M: = \textbf{H}^{3}/G$ is finite to one on $U_i$.

Now we argue by way of contradiction. Let $\mathcal {Q}_N :
\textbf{H}^{3} \rightarrow N $ and let $\mathcal {Q}_M :
\textbf{H}^{3} \rightarrow M$ be the covering maps and notice that
$\mathcal {Q}_M=\mathcal {P} \circ \mathcal {Q}_N$. Suppose that $[G:H]$
is infinite. This implies that $\mathcal {P}$ is an infinite
covering map. By the definition of a geometrically infinite end, we
can take a point $z$ from the neighbourhood $U_1$ of a geometrically
infinite end $E_1$ such that $z$ also lies in the convex core of
$N$. By lifting the point $\mathcal {P}(z) \in M$ to
$\textbf{H}^{3}$, it is easy to see that the infinite set
$\widetilde{S}: = \mathcal {Q}_M^{-1}(\mathcal {P}(z))$ lies in the
common convex hull $\widetilde{C}(H) = \widetilde{C}(G)\subset
\textbf{H}^{3}$, by observing that $\widetilde{C}(G)$ is invariant
under $G$ and the preimage $\mathcal {Q}_N^{-1}(z) \subset
\widetilde{S}$ lies in $\widetilde{C}(H)$. Since $\mathcal {P}$ is
an infinite covering map, the set $S: = \mathcal
{P}^{-1}(\mathcal {P}(z))$ is infinite.  By considering $\mathcal
{Q}_M=\mathcal {P} \circ \mathcal {Q}_N$, it follows that $S =
\mathcal {Q}_N(\widetilde{S}) \subset N$ and thus $S$ lies in the
convex core of $N$, because $\widetilde{S} \subset
\widetilde{C}(H)$.

We claim that we can take a smaller invariant horoball system for
$H$ such that infinitely many points of $S$ lie outside all cusp
ends of $N$. Otherwise, we can suppose that infinitely many points
of $S$ are contained inside a cusp end $E_c$ of $N$, since there are
only finitely many cusp ends for $N$. Thus infinitely many points of
$\mathcal {Q}_N^{-1}(S)$ lie in the corresponding horoball $B$ for
the end $E_c$. Normalizing the parabolic fixed point for $E_c$ to
$\infty$ in the upper half space model of $\textbf{H}^{3}$, the
horoball $B$ at $\infty$ is precisely invariant under the stabilizer
of $\infty$ in $H$. On the other hand, we have that infinitely many
points of $\mathcal {Q}_N^{-1}(S)$ have the same height, since the
covering map $\mathcal {P}$ maps $S \subset N$ to a single point on
$M$, and the horoball $B$ is also precisely invariant under the
stabilizer of $\infty$ in $G$, which is a Euclidean group preserving
the height of points in the horoball $B$. Then we can take a smaller
horoball for $E_c$ such that these infinitely many points of
$\mathcal {Q}_N^{-1}(S)$ lie outside the horoball.

Continuing the above process for all cusp ends of $N$, we can get a new
invariant horoball system such that infinitely points of $S$ lie
outside these cusp ends of $N$.

Since $S$ projects to a single point on $M$, we can conclude that
$S$ cannot lie in any compact subset of the convex core of $N$. Thus
by the above second claim, there exist infinitely many points of $S$
which can only lie in geometrically infinite ends of $N$. This is a
contradiction to the Covering Theorem \ref{thm:covering},  which
states that the covering map $\mathcal {P}$ restricted to each
geometrically infinite end is finite-to-one, if $H$ is not a
virtually fiber subgroup of $G$.
\end{proof}

\acks The first author would like to sincerely thank Prof. Jim
Anderson for helpful communications and Prof. Leonid Potyagailo for
many helpful comments and discussions about the original version.
The authors would like to express their gratitude to the referee for
his/her many helpful comments, which substantially improved the
exposition of the paper.







\end{document}